\documentclass[12pt]{amsart}
\usepackage{tikz,amssymb,enumerate}
\usepackage[margin=2.5cm]{geometry}
\allowdisplaybreaks

\newtheorem{theorem}{Theorem}[section]
\newtheorem*{theoremA}{Theorem A}
\newtheorem*{theoremB}{Theorem B}

\newtheorem{lemma}[theorem]{Lemma}
\newtheorem{corollary}[theorem]{Corollary}
\theoremstyle{definition}
\newtheorem{definition}[theorem]{Definition}
\theoremstyle{remark}

\newcommand{\Hom}{\operatorname{Hom}}
\newcommand{\Ind}{\operatorname{Ind}}
\newcommand{\Span}{\operatorname{span}}
\newcommand{\conv}{\operatorname{conv}}

\newcommand{\C}{\mathbb{C}}
\newcommand{\R}{\mathbb{R}}
\newcommand{\N}{\mathbb{N}}
\newcommand{\group}[1]{\mathrm{#1}}
\newcommand{\Lie}[1]{\mathfrak{#1}}
\newcommand{\half}{\tfrac{1}{2}}
\newcommand{\Hil}{\mathcal{H}}
\newcommand{\cH}{\mathcal{H}}
\newcommand{\Alpha}{\mathrm{A}}

\newcommand{\norm}[1]{\left\Vert #1 \right\Vert}
\newcommand{\lnorm}{\left\Vert}
\newcommand{\rnorm}{\right\Vert}
\newcommand{\lip}{\left\langle}
\newcommand{\rip}{\right\rangle}
\newcommand{\lset}{\left\lbrace}
\newcommand{\rset}{\right\rbrace}
\newcommand{\lpar}{\left(}
\newcommand{\rpar}{\right)}
\newcommand{\bigglpar}{\biggl(}
\newcommand{\biggrpar}{\biggr)}
\newcommand{\labs}{\left\vert}
\newcommand{\rabs}{\right\vert}
\newcommand{\bigglabs}{\biggl\vert}
\newcommand{\biggrabs}{\biggr\vert}

\def\fn{\ignorespaces\,}
\newcommand{\ds}{\displaystyle}
\def\wrt{\ignorespaces\,\mathrm{d}}
\newcommand{\expe}{\mathrm{e}}

\def\rist{\bigr|_}
\newcommand{\afterbar}[1][x]{\overline{\phantom{#1}}}
\newcommand{\calH}{\mathcal{H}}

\newcommand{\fnspace}[1]{\mathsf{#1}}

\newcommand{\ad}{\operatorname{ad}}

\newcommand{\trace}{\operatorname{trace}}
\renewcommand{\Re}{\operatorname{Re}}
\renewcommand{\Im}{\operatorname{Im}}
\newcommand{\temp}{\mathrm{temp}}
\newcommand{\Her}{\mathrm{Her}}

\newcommand{\MQ}{M_Q}

\title[Decay estimates for matrix coefficients]{Decay estimates for matrix coefficients \\ of unitary representations  \\ of semisimple Lie groups}
\author{Michael G Cowling}
\address{School of Mathematics and Statistics, UNSW Sydney, Sydney NSW 2052, Australia}
\email{m.cowling@unsw.edu.au}
\thanks{The author was partially supported by the Alexander von Humboldt Stiftung and the Australian Research Council DP220100285.}
\thanks{Various individuals answered or asked interesting questions.
In particular thanks are due to Jean-Philippe Anker, Jeff Adams, Roger Howe, Tony Knapp, David Vogan, and Bob Zimmer.
Thanks are also due to the editor handling the paper and the anonymous referee for very helpful and constructive criticism.}

\begin{document}

\begin{abstract}
Let $G$ be a connected semisimple Lie group with finite centre and $K$ be a maximal compact subgroup thereof.
Given a function $u$ on $G$, we define $\mathcal{A} u$ to be the root-mean-square average over $K$, acting both on the left and the right, of $u$.
We take a positive-real-valued spherical function $\phi_{\lambda}$ on $G$, and study the Banach convolution algebra of $\fnspace{C}_c(G)$-functions $u$ with the norm $\norm{u}_{(\lambda)} := \int_G \mathcal{A} u(x) \fn \phi_{\lambda}(x) \,dx$.
The $C^*$ completion of this algebra is an exotic $C^*$-algebra on $G$, in the sense that it lies ``between'' the reduced $C^*$-algebra of $G$ and the full $C^*$-algebra of $G$, and in the sense that it arise as the completion of a star-algebra that does not contain an approximate identity.

Using functional analysis and representation theory, we show that for all unitary representations $\pi$ of $G$, there exists a unique minimal positive-real-valued spherical function $\phi_{\lambda}$ on $G$ such that $\mathcal{A} \lip \pi(\cdot) \xi, \eta\rip  \leq \lnorm \xi \rnorm_{\Hil_\pi} \lnorm \eta \rnorm_{\Hil_\pi}    \phi_{\lambda}$.
This estimate has nice features of both asymptotic pointwise estimates and Lebesgue space estimates; indeed it is equivalent to pointwise estimates
$\labs \lip \pi(\cdot) \xi, \eta\rip \rabs \leq C(\xi, \eta)    \fn\phi_{\lambda}$
for $K$-finite or smooth vectors $\xi$ and $\eta$, and it exhibits different decay rates in different directions at infinity in $G$.
Further, if we assume the latter inequality with  arbitrary $C( \xi, \eta )$, we can prove the former inequality and then return to the latter inequality with explicit knowledge of $C( \xi, \eta )$.
Finally, it holds everywhere in $G$, in contrast to asymptotic estimates which are not global and to $\fnspace{L}^p$ estimates which carry no pointwise information.

\end{abstract}

\maketitle

\section{Introduction}

Throughout in this paper, a \emph{semisimple Lie group} means a noncompact connected real semisimple Lie group $G$ with finite centre.
However, because the methods are quite abstract, the ideas apply more generally, and in particular to the $p$-adic case and to the more general reductive groups considered by Harish-Chandra to enable inductive arguments.

To present our main results, we need some notation; details may be found below.
Let $K$ be a maximal compact subgroup of a semisimple Lie group $G$.
Then $G$ has a Cartan decomposition $K A^+ K$, where $A^+$ is a cone in a simply connected abelian subgroup $A$ of $G$.
We write $\Lie{a}^+$ for the corresponding cone in the Lie algebra $\Lie{a}$ of $A$.
We take $\rho$ to be a particular element of $\Lie{a}^* := \Hom(\Lie{a},\R)$;
each $\lambda$ in $\Lie{a}^*_{\C} := \Hom(\Lie{a},\C)$ induces a homomorphism $\exp H \mapsto \expe^{\lambda(H)}$ from $A$ to $\C$.

Here and later, $\wrt y$ denotes the element of Haar measure on the group over which $y$ varies.

A $K$-bi-invariant function $\phi$ on $G$ is said to be a (zonal) spherical function if and only if
\begin{equation}\label{eq:spherical-functions-0}
\int_K \phi(xky) \wrt k = \phi(x) \fn \phi(y)
\qquad\forall x, y  \in G,
\end{equation}
and a spherical function $\phi$ on $G$ is said to be hermitean if $\phi(x^{-1}) = \bar{\phi}(x)$ for all $x \in G$.
Harish-Chandra gave an integral formula for spherical functions $\phi_\lambda$, where $\lambda \in \Lie{a}^*_{\C}$; these exhaust the spherical functions on $G$.
There is a closed subset $\Lie{a}^{*,\Her}$ of $\Lie{a}^*_{\C}$, which we describe later, such that $\phi_\lambda$ is positive-real-valued and hermitean if and only if $\lambda \in \Lie{a}^{*,\Her}$.
For $\lambda \in (\Lie{a}^*)^+$, the function $\phi_\lambda(\exp H)$ is comparable in the whole of $(\Lie{a}^*)^+$ to the product of an exponential $\expe^{(\lambda-\rho)(H)}$ and a polynomial $p_\lambda(H)$; see Theorem \ref{thm:NPP} for details.

Recall that, given unitary representations $\sigma$ and $\pi$ of $G$, we say that $\sigma$ is weakly contained in $\pi$, written $\sigma \preceq \pi$, if and only if all positive definite functions associated to $\sigma$ may be approximated uniformly on compacta in $G$ by positive definite functions associated to $\pi$.

Given a continuous function $u$ on $G$, we define its \emph{root-mean-square average} $\mathcal{A} u$:
\[
\mathcal{A} u (x)
:= \lpar \int_K \int_K  \labs u(k x k') \rabs^2 \wrt k \wrt k' \rpar^{1/2}
\qquad\forall x \in G.
\]
We now state our main theorems, which were inspired by \cite{Hz70} and \cite{SW18}.

\begin{theoremA}
Let $\pi$ be a unitary representation of a semisimple Lie group $G$, and suppose that $\lambda \in \Lie{a}^{*}$.
Consider the condition
\begin{equation}\label{eq:strong-assumption-1}
\mathcal{A} \lip \pi(\cdot) \xi, \eta\rip
\leq \lnorm \xi \rnorm_{\Hil_\pi} \lnorm \eta \rnorm_{\Hil_\pi}
    \phi_\lambda
\qquad\forall \xi, \eta \in \cH_\pi.
\end{equation}
Then \eqref{eq:strong-assumption-1} holds if and only if all $\sigma \in \bar{G}$ such that $\sigma \preceq \pi$ satisfy \eqref{eq:strong-assumption-1} with $\pi$ replaced by $\sigma$.
\end{theoremA}

In particular, the inequality \eqref{eq:strong-assumption-1} passes to the irreducible representations of $G$ that appear in the direct integral decomposition of $\pi$.

\begin{theoremB}
Let $\pi$ be a unitary representation of a semisimple Lie group $G$, and suppose that $\lambda \in \Lie{a}^{*,\Her}$.
Then the following are equivalent:
\begin{equation}\label{eq:cond-1-1}
\mathcal{A} \lip \pi(x) \xi, \eta\rip
\leq \lnorm \xi \rnorm_{\Hil_\pi} \lnorm \eta \rnorm_{\Hil_\pi}
    \phi_\lambda(x)
\end{equation}
for all $x \in G$ and for all $\xi$ and $\eta$ in $\cH_\pi$, and
\begin{equation}\label{eq:cond-2-1}
\sup_{k,k' \in K} \labs \lip \pi (k x k') \xi, \eta\rip\rabs  \leq C(\xi, \eta) \fn\phi_\lambda(x)
\end{equation}
for all $x \in G$ and all $k,k' \in K$, and for all $K$-finite $\xi, \eta \in \cH_\pi$.
Further, if these conditions hold, then we may take $C(\xi, \eta) $ to be given by
\[
C(\xi, \eta) = \dim(\Span(\pi(K)\xi))^{1/2} \lnorm \xi \rnorm_{\Hil_\pi} \dim(\Span(\pi(K)\eta))^{1/2}  \lnorm \eta \rnorm_{\Hil_\pi},
\]
and further, \eqref{eq:cond-2-1} holds for all smooth $\xi, \eta \in \cH_\pi$.

If moreover $\pi$ is irreducible, then conditions \eqref{eq:cond-1-1} and \eqref{eq:cond-2-1} both hold provided that there exist $\xi$ and $\eta$ in $\cH_\pi \setminus \{0\}$
such that
\begin{equation}\label{eq:cond-3-1}
\mathcal{A} \lip \pi(x) \xi, \eta\rip
\leq C(\xi,\eta)  \fn\phi_\lambda(x)
\qquad\forall x \in G.
\end{equation}
\end{theoremB}

From these theorems, it follows that
\begin{enumerate}
  \item[(a)] if $\pi$ is an irreducible unitary representation, then an estimate of the form \eqref{eq:strong-assumption-1} holds, where $\lambda \in \Lie{a}^{*,\Her}$ and $\phi_\lambda$ has the same exponential decay at infinity as the $K$-finite matrix coefficients of $\pi$, and
  \item[(b)] if $\pi$ is a general unitary representation, then an estimate of the form \eqref{eq:strong-assumption-1} holds, and there is an optimal $\phi_\lambda$ where $\lambda \in (\Lie{a}^*)\afterbar$.
      Unless $\pi$ weakly contains the trivial representation,
      $\phi_\lambda(\exp H)$ decays exponentially at infinity in $H$.
\end{enumerate}

Sharp estimates for the spherical functions $\phi_\lambda$, where $\lambda \in \Lie{a}^*$, due to Narayanan, Pasquale and Pusti \cite{NPP14}, may be found in Theorem \ref{thm:NPP} below.

\section{Notation and Background}

In order to go into more detail and prove our results, we need quite a lot of standard notation.
If there is no explicit reference given, the results on abstract harmonic analysis and $C^*$-algebras may be found in standard texts such as \cite{HR63} or \cite{Di60} and those on Lie groups and spherical functions may be found in \cite{He84} or \cite{Kn86}.

\subsection{Unitary representations and $\fnspace{B}(G)$}\label{ssec:B-of-G}

In this section, we let $G$ be a locally compact group.
For (suitable) functions $f$ and $f'$ on $G$, we write $f * f'$ for their usual convolution and $f^*$ for the function $x \mapsto \Delta(x)^{-1} \bar{f}(x^{-1})$, where $\Delta$ is the modular function.

Denote by $\bar G$ the ``set'' of all continuous unitary representations $\pi$ of $G$ on Hilbert spaces $\cH_\pi$, and by $\hat G$ the subset of $\bar G$ consisting of irreducible representations, all modulo unitary equivalence.
(We write ``set'' to point out that care is required; treating this as a set can lead to some unedifying problems in set theory.)

For any unitary representation $\pi$ of $G$ and $f \in \fnspace{L}^1(G)$, we define the operator $\pi(f)$:
\[
\pi(f) := \int_{G} f(x) \fn \pi(x) \wrt x;
\]
the integral converges in the weak operator topology.
Then $\pi(f*f') = \pi(f) \pi(f')$ and $\pi(f^*) = \pi(f)^*$, so that $\pi(\fnspace{L}^1(G))$ is a star-algebra of operators on $\calH_\pi$.

A \emph{matrix coefficient of a unitary representation $\pi$ of $G$} is a function $u$ of the form $\lip \pi(\cdot) \xi, \eta \rip$, that is,
\begin{equation*}
u(x) = \lip \pi(x) \xi, \eta \rip
\quad\forall x \in G,
\end{equation*}
where $\pi\in \bar G$ and  $\xi, \eta \in \cH_\pi$.
We abbreviate this formula to $u = \lip \pi(\cdot) \xi, \eta \rip$.

The \emph{Fourier--Stieltjes algebra} $\fnspace{B}(G)$ is the space of all matrix coefficients of all unitary representations:
\[
\fnspace{B}(G) = \{ u \in C(G) : u = \lip\pi(\cdot) \xi,\eta\rip, \  \pi \in \bar G, \  \xi, \eta \in \cH_\pi \} ;
\]
the same function $u$ may arise in different ways.
Pointwise addition and multiplication of matrix coefficients correspond to direct sums and tensor products of representations.
Finally, we norm $\fnspace{B}(G)$: for $u \in \fnspace{B}(G)$,
\[
\lnorm u \rnorm_\fnspace{B}
:= \min\{ \lnorm \xi\rnorm \lnorm \eta\rnorm :  u = \lip \pi(\cdot) \xi, \eta \rip, \  \pi \in \bar G, \  \xi, \eta \in \cH_\pi \}.
\]
The theory of $C^*$-algebras shows that the minimum is attained.

The set of matrix coefficients of a fixed representation $\pi$ is not \emph{a priori} a vector space, and to allow us to use the tools of functional analysis, we consider the closed linear span of the matrix coefficients in $\fnspace{B}(G)$.
Following Arsac \cite{Ar76}, we define $\fnspace{A}_\pi(G)$ to be the linear space of all infinite linear combinations
\[
\sum_{j\in\N} \lip \pi(\cdot) \xi_j, \eta_j \rip
\qquad\text{such that}\qquad
\sum_{j\in\N} \norm{\xi_j}_{\calH_\pi} \norm{\eta_j}_{\calH_\pi} < \infty;
\]
we norm $\fnspace{A}_\pi(G)$ by setting
\[
\norm{u}_{\fnspace{A}_\pi}
:= \inf\lset \sum_{j\in\N} \norm{\xi_j}_{\calH_\pi} \norm{\eta_j}_{\calH_\pi} : u = \sum_{j\in\N} \lip \pi(\cdot) \xi_j, \eta_j \rip \rset.
\]
Then $\fnspace{A}_\pi(G)$ may be identified with the space of all functions of the form
\[
u(x) = \trace( T \pi(x))
\qquad\forall x \in G,
\]
where $T$ is a trace class operator on $\cH_\pi$, and $\norm{u}_{\fnspace{A}_\pi}$ coincides with the trace norm of $T$.
Alternatively, $\fnspace{A}_\pi(G)$ may be identified with the space of all matrix coefficients of the infinite sum of copies of $\pi$.
The theory of $C^*$-algebras shows that $\lnorm u \rnorm_{\fnspace{A}_\pi}  = \lnorm u \rnorm_\fnspace{B}$ for all $u \in \fnspace{A}_\pi(G)$.
See \cite{Ey64} for more on this.

Weak containment of group representations was developed by Fell \cite{Fe60} and then Arsac \cite{Ar76}, and the following theory is contained in their work.
A unitary representation $\sigma$ of a locally compact group $G$ is \emph{weakly contained} in a unitary representation $\pi$ of $G$ if any one of the following equivalent conditions holds:
\begin{enumerate}
  \item[(a)] $\norm{\sigma(f)} \leq \norm{\pi(f)}$ (the norms here are the operator norms on $\calH_\sigma$ and $\calH_\pi$);
\item[(b)] every positive definite function $\lip\sigma(\cdot)\xi,\xi\rip$ (where $\xi \in \calH$) associated to $\sigma$ is the locally uniform limit of sums $\sum_j\lip\pi(\cdot)\xi_j,\xi_j\rip$ of positive definite functions associated to $\pi$;
\item[(c)] every $u \in \fnspace{A}_\sigma(G)$ is the locally uniform limit of $v \in \fnspace{A}_\pi(G)$ such that $\norm{v}_\fnspace{B} \leq \norm{u}_\fnspace{B}$.
\end{enumerate}
We write $\sigma \preceq \pi$ if $\sigma$ is weakly contained in $\pi$, and say that representations are weakly equivalent if each is weakly contained in the other.
Then every unitary representation $\pi$ is weakly equivalent to the direct sum of all irreducible unitary representations $\sigma$ that are weakly contained in $\pi$.

\subsection*{Discrete series}
Let $\pi$ be an irreducible unitary representation of a locally group $G$.  Then $\pi$ is unitarily equivalent to a subrepresentation of the regular representation of $G$ if and only if one of its nontrivial matrix coefficients lies in $\fnspace{L}^2(G)$, or equivalently, all its matrix coefficients lie in $\fnspace{L}^2(G)$.
When any and hence all of these conditions hold, we say that $\pi$ belongs to the \emph{discrete series} of representations of $G$.

\subsection*{Tempered representations}
We say that a representation $\sigma$ of a locally compact group $G$ is \emph{tempered} if it is weakly contained in the regular representation of $G$.

If $H$ is a closed subgroup of a locally compact group $G$, then we define $\bar{G}_{H,\temp}$ to be the subcollection of $\bar{G}$ of $H$-tempered unitary representations of $G$, that is, the representations of $G$ whose restrictions to $H$ are tempered on $H$.
The following lemma is obvious.

\begin{lemma}\label{lem:h-temp}
Let $H$ be a closed subgroup of a locally compact group $G$ and $\bar{G}_{H,\temp}$  be as defined above.
The collection $\bar{G}_{H,\temp}$ representations is closed under direct sums and integrals, and under taking tensor products with arbitrary unitary representations of $G$.
Every unitary representation of $G$ that is weakly contained in an $H$-tempered representation of $G$ is also $H$-tempered.
If $H_1 \subseteq H_2$, then $\bar{G}_{H_1,\temp} \supseteq \bar{G}_{H_2,\temp}$.
\end{lemma}

\subsection*{Structure of semisimple Lie groups}
Lie groups and their Lie algebras are denoted by upper case italic letters and the corresponding lower case fraktur letter.

Let $G$ be a semisimple Lie group.
Its Lie algebra $\Lie{g}$ admits a Cartan involution $\theta$, that is, a nontrivial Lie algebra isomorphism such that $\theta^2 = 1$, and $\Lie{g}$ splits as a direct sum $\Lie{k} \oplus \Lie{p}$, where $\Lie{k}$ is the $+1$ eigenspace and $\Lie{p}$ is the $-1$ eigenspace; $\Lie{k}$ is a subalgebra but $\Lie{p}$ is not.

Let $\Lie{a}$ be a maximal commutative subspace of $\Lie{p}$, of dimension $r$, say; then the endomorphisms $\ad(H)$ of $\Lie{g}$ are simultaneously diagonalisable.
For $\alpha \in \Lie{a}^*$, we define
\[
\Lie{g}_\alpha = \{ X \in \Lie{g} : \ad(H) X = \alpha(H) X \}.
\]
A \emph{(restricted) root} $\alpha$ is a nonzero element of $\Lie{a}^*$ such that $\Lie{g}_\alpha \neq \{0\}$.
Then
\[
\Lie{g} = \Lie{g}_0 \oplus \sum_{\alpha \in \Sigma} \Lie{g}_\alpha,
\]
where $\Sigma$ is the root system formed by all the roots.

The inner product $\lip X, Y \rip := \trace(\ad (X) \ad (\theta Y))$ on $\Lie{g}$ enables us to identify $\Lie{a}$ with $\Lie{a}^*$ and to put an inner product on $\Lie{a}^*$.
Each element $\lambda$ of $\Lie{a}^*$ or $\Lie{a}^*_{\C}$ gives rise to a homomorphism  from $A$ to $\C\setminus\{0\}$:
\[
\exp H \mapsto \expe^{\lambda (H)}.
\]

The hyperplanes $\{ H \in \Lie{a} : \alpha(H) = 0 \}$, where $\alpha \in \Sigma$, divide $\Lie{a}$ into \emph{Weyl chambers}.
We pick an arbitrary chamber, and say that it is positive; we write it as $\Lie{a}^+$.
We equip $\Lie{a}^*$ with a partial order, by defining $\beta \leq \gamma$ if and only if $\beta(H) \leq \gamma(H)$ for all $H \in \Lie{a}^+$.
We write $\Sigma^+$ for the set of positive roots, that is, the roots $\alpha$ such that $0 \leq \alpha$.
We may choose a set $\Delta$ of \emph{simple positive roots} in $\Sigma^+$ such that every positive root is a sum, with nonnegative integer coefficients, of roots in $\Delta$.
Let $\rho = \half \sum_{\alpha \in \Sigma^+} \dim(\Lie{g}_\alpha) \alpha$; then $\rho \in (\Lie{a}^*)^+$, the cone in $\Lie{a}^*$ corresponding to $\Lie{a}^+$ under the identification of $\Lie{a}$ and $\Lie{a}^*$.

The \emph{Weyl group} $W$ is the finite group of orthogonal transformations of $\Lie{a}$ generated by the reflections in the hyperplanes $\{ H \in \Lie{a} : \alpha(H) = 0 \}$, where $\alpha \in \Sigma$; it acts simply transitively on the set of Weyl chambers.
There is an obvious induced action on $\Lie{a}^*$.
In some cases, $W$ contains $-I$ (where $I$ denotes the identity) and in some cases it does not.
When $G$ is simple, $W$ does not contain $-I$ when the root system is of type $A_r$ (where $r \geq 2$), $D_r$ (where $r$ is odd), and $E_6$; see \cite[Exercise 5, page 71]{Hu72}.

We define $\Lie{n} := \sum_{\alpha \in \Sigma^+} \Lie{g}_\alpha$.
This is a Lie algebra because $[ \Lie{g}_\alpha, \Lie{g}_\beta] \subseteq \Lie{g}_{\alpha+\beta}$.
The subgroup $K$ of $G$ with Lie algebra $\Lie{k}$ is a maximal compact subgroup of $G$, and is connected; the subgroup $A$ of $G$ whose Lie algebra is $\Lie{a}$ is a maximal simply connected abelian subgroup; the subgroup $N$ of $G$ whose Lie algebra is $\Lie{n}$ is nilpotent.
We write $A^+$ for the image of $\Lie{a}^+$ in $A$ under the exponential map, $(A^+)\afterbar$ for its closure in $A$, and $M$ for the centraliser of $A$ in $K$.
The groups $M$ and $A$ normalise $N$.

The continuous map $(k,a,k') \mapsto kak'$ from $K \times (A^+)\afterbar \times K$ to $G$ is surjective, that is, every element $x$ of $G$ may be written in the form
\[
x = kak'
\]
where $k, k' \in K$ and $a \in (A^+)\afterbar{}$.
If $x \in K$, then $a$ is the identity, and there is lots of ambiguity in the choice of $k$ and $k'$; otherwise, there is less ambiguity.
In particular, for elements of the form $kak'$ where $a\in A^+$, the element $a$ is uniquely determined while $kak' = k''ak'''$ if and only if there exists $m \in M$ such that $k'' = km$ and $k' = mk'''$.
This is known as the \emph{Cartan decomposition}.
A Haar measure on $G$ is given by
\begin{equation}\label{eq:Haar-measure}
\int_G f(x) \wrt x
= \int_K \int_{\Lie{a}^+} \int_K f(k \exp(H)k') \fn \mathrm{w}(H) \wrt k  \wrt H \wrt k' ,
\end{equation}
where $\mathrm{w}(H) = \prod_{\alpha\in\Sigma^+} (\sinh \alpha(H))^{\dim\Lie{g}_\alpha}$, which may be rewritten as a weighted sum of exponentials.
The dominant term, that is, the biggest exponential on $\Lie{a}^+$, is $\expe^{2\rho (H)}$.

The map $(k,a,n) \mapsto kan$ from $K \times A \times N$ to $G$ is a diffeomorphism.
Hence we may write each $x$ in $G$ uniquely as a product $k(x) a(x) n(x)$, where $k(x) \in K$, $a(x) \in A$ and $n(x) \in N$.
This is known as the \emph{Iwasawa decomposition} of $G$.
The Haar measure on $G$ may be written in terms of the Haar measures on $K$, $A$ and $N$.

The subgroup $MAN$ of $G$ is called a \emph{minimal parabolic subgroup}, and any closed subgroup of $G$ that contains $MAN$ is called a \emph{parabolic subgroup}.
Every such subgroup $Q$ may be parametrised by a subset $\Delta(Q)$ of the set $\Delta$ of simple positive roots:
at the Lie algebra level, $\Lie{q} = \Lie{m} + \Lie{a} + \Lie{n} + \sum_{\alpha \in \Span(\Delta(Q))} \Lie{g}_{\alpha}$; we may write $Q$ as $M_Q A_QN_Q$ where $M_Q$ is reductive and contains $M$, $A_Q \subseteq A$ and $N_Q \subseteq N$.
Further, $M_Q$ and $A_Q$ commute, and both normalise $N_Q$.
Since $Q \supseteq P$, every element of $G$ admits a (not necessarily not unique) decomposition of the form $kq$, where $k \in K$ and $q\in Q$.
For more details, see, for instance, \cite[Section V.5]{Kn86}.

\subsection*{Unitary representations}

We recall that $K$ has a maximal torus $T$, and the irreducible unitary representations of $K$ are parametrised by a lattice in $\Lie{t}^*$, modulo the action of a Weyl group.
The inner product on $\Lie{g}$ induces inner products and norms on $\Lie{k}$ and $\Lie{t}$ and their dual spaces; we write $\norm{\tau}$ for the norm of the parameter in $\Lie{t}^*$ of $\tau$ in $\hat{K}$.
We need two standard results of analysis on compact Lie groups.
See, for instance, \cite[Theorem 7.43 and Proposition 10.6]{Ha15}.

First, the (negative of the) Laplace--Beltrami operator $\Delta_K$ on $K$ is a canonical second order elliptic operator that acts by a scalar $C(\tau)$ on the space of matrix coefficients of the irreducible representation $\tau$ of $K$, and
\begin{equation}\label{eq:Lap-Bel-eigenvalue}
c_1 \norm{\tau}^2 \leq C(\tau) \leq c_2 \norm{\tau}^2  .
\end{equation}
Second, the Weyl dimension formula shows that
\begin{equation}\label{eq:Weyl-dim-formula}
\dim(\cH_\tau) \leq c_3 \norm{\tau}^{(\dim(K) - \dim(T))/2} .
\end{equation}
Here $c_1$, $c_2$ and $c_3$ are constants that depend on $K$.

If $\pi \in \bar G$, then $\pi \rist{K}$, its restriction to $K$, decomposes as a sum of irreducible representations of $K$.
That is, $\cH_\pi = \bigoplus_{\tau \in \hat K} n_\tau \cH_\tau$, and $\pi \rist{K} = \bigoplus_{\tau \in \hat K} n_\tau \tau$.
Let $P_\tau$ be the orthogonal projection of $\cH_\pi$ onto $n_\tau \cH_\tau$.
We say that $\xi \in \cH_\pi$ is \emph{$\tau$-isotypic} if $P_\tau \xi = \xi$, and \emph{$K$-finite} if it is a finite linear combination of isotypic vectors.
In general, the multiplicities $n_\tau$ may be infinite; however, if $\xi$ is $\tau$-isotypic, then the dimension of the space spanned by $\{ \pi(x)\xi : x \in K \}$ is no more than $\dim(\cH_\tau)^2$, as every cyclic representation of a compact group is equivalent to a subrepresentation of the regular representation of the compact group.

A $K$-finite matrix coefficient is a matrix coefficient $\lip \pi(\cdot) \xi, \eta \rip$ where both the vectors are $K$-finite.
The span of the set  of left and right translates by $K$ of a $K$-finite matrix coefficient is finite-dimensional.
If $\cH_\pi^0$ is a dense $\pi(K)$-invariant subspace of $\cH_\pi$, then the subspace of $\cH_\pi^0$ of $K$-finite vectors in $\cH_\pi^0$ is also a dense $\pi(K)$-invariant subspace of $\cH_\pi$.
In general, if $\xi$ is $K$-finite and $x \in G$, then $\pi(x)\xi$ need not be $K$-finite.

A vector $\xi$ in $\cH_\pi$ is said to be \emph{smooth} if the mapping $x \mapsto \pi(x) \xi$ is infinitely differentiable.
If $\xi$ is smooth and $x \in G$, then $\pi(x)\xi$ is also smooth.
The smooth vectors form a dense subspace of $\cH_\pi$.
By identifying the restriction of $\pi$ to the closure of the space spanned by $\{ \pi(x)\xi : x \in K \}$ with a subrepresentation of the regular representation, and using \eqref{eq:Lap-Bel-eigenvalue} and the fact that $\Delta_K^N \xi \in \cH_\pi$ for all $N \in \N$, we see that it is possible to write
\begin{equation}\label{eq:xi-decomp}
\xi =  \bigoplus_{\tau \in \hat{K}} \xi_\tau,
\end{equation}
where $\xi_\tau$ is $\tau$-isotypic and
\begin{equation}\label{eq:smooth-decomp}
\norm{\xi_\tau}_{\cH_\pi} = O(\norm{\tau}^{-N})
\qquad\forall \tau \in \hat{K}.
\end{equation}
As noted above, the $K$-finite smooth vectors are dense in the space of smooth vectors.

\subsection*{Irreducible unitary representations}
For irreducible representations $\pi$ of a semisimple Lie group $G$, we can be more precise:
Harish-Chandra's subquotient theorem (which may be sharpened to a subrepresentation theorem) implies that $n_\tau \leq \dim(\cH_\tau)$ for all
$\tau \in \hat K$.
Similarly, if $\xi$ is $\tau$-isotypic, then $\dim(\Span \pi(K) \xi) \leq \dim(\cH_\tau)$.

Take $\pi \in \hat G$ and $\sigma , \tau \in \hat K$.
Define the smooth function $\Phi_{\sigma,\tau}: A \to \Hom(n_\sigma \cH_\sigma, n_\tau \cH_\tau)$ by
\[
\Phi_{\sigma,\tau}(a) =  P_\tau \pi(a) P_\sigma
\quad\forall a \in A .
\]
If $\xi$ is $\sigma$-isotypic and $\eta$ is $\tau$-isotypic, then
\[
\begin{aligned}
\lip\pi(k a k') \xi,\eta\rip
&=  \sum_{i,j} \lip \theta_i, \pi(k') \xi\rip
\lip \zeta_j, \pi(k^{-1}) \eta\rip \lip \Phi_{\sigma,\tau}(a) \theta_i , \zeta_j \rip \\
&=  \sum_{i,j} \lip \theta_i, \sigma(k') \xi\rip
\lip \tau(k)\zeta_j,  \eta\rip \lip \Phi_{\sigma,\tau}(a) \theta_i , \zeta_j \rip
\end{aligned}
\]
for all $k, k' \in K$ and all $a \in A$, where the $\theta_i$ and $\zeta_j$ form orthonormal bases for the spaces $n_\sigma \cH_\sigma$ and $n_\tau \cH_\tau$.
Thus the collection of matrix-valued functions $\Phi_{\sigma,\tau}$  encapsulates the behaviour of $K$-finite matrix coefficients of $\pi$.

Now we consider the behaviour of matrix coefficients of irreducible representations at infinity.
The following theorem is essentially contained in \cite[Chapter VIII, Section 8]{Kn86}.

\begin{theorem}\label{thm:asymptotics}
Suppose that $\pi \in \hat G$ and $\sigma , \tau \in \hat K$.
Then
\begin{equation}\label{eq:asymptotics-sum}
\Phi_{\sigma,\tau} = \sum_{ l} \Phi_{\sigma,\tau, l} ,
\end{equation}
where for each $l$ there exists a \emph{leading exponent} $\lambda_l \in \Lie{a}^*_\C$ and a nontrivial polynomial $p_l$, independent of $\sigma$ and $\tau$, and a function $\phi_l: \Lie{a} \to \Hom(n_\sigma \cH_\sigma, n_\tau \cH_\tau)$ such that
\begin{equation}\label{eq:asymptotics}
\Phi_{\sigma,\tau,l} (\exp H)
= p_l (H) \fn\expe^{ (\lambda_l - \rho) H} \fn\phi_{l}(H)
\quad\text{as $H \to \infty$}.
\end{equation}
Here $\phi_{l}(H)$ may be written as a convergent power series near $0$ in the variables $\expe^{-\alpha_1(H)}$, \dots, $\expe^{-\alpha_r(H)}$ with nonzero constant term.
The indices $l$ may be chosen such that $\Re \lambda_{1}$ lies in the closure of  $(\Lie{a}^*)^+$, and each $\lambda_l$ is of the form $w\lambda_1$ for some $w \in W$; the $\lambda_l$ do not coincide.
There is a fixed $N$, which depends on $G$, such that $\deg p_l \leq N$.
\end{theorem}

In the theorem, the expression $H \to \infty$ means that $|H| \to \infty$ and $H$ \emph{stays away from the walls of the Weyl chamber}, that is, $H$ is constrained to lie in a proper open subcone $\Lie{c}$ of $\Lie{a}^+$.
It is easier to handle many analytic phenomena in $G$ when one does this.
Similarly, when $\lambda_1$ does not lie on a wall of the Weyl chamber, there are exactly $|W|$ summands in \eqref{eq:asymptotics-sum}, and the polynomial terms are constants.
When $\pi$ is tempered, the $\lambda_l$ are purely imaginary or negative.
We are interested in nontempered representations; in this case, some of the exponential terms decay much faster than others.
The parameters $\lambda_l$ control the decay rate of the $K$-finite matrix coefficients.

The asymptotic expansion was used (particularly by Langlands and Knapp) to provide the following partial \emph{Langlands classification} of irreducible unitary representations of a semisimple Lie group $G$.
See \cite[Chapter XIV, Section 17]{Kn86} for more details.

\begin{theorem}\label{thm:Langlands-Knapp-clsfctn}
Suppose that $\pi$ is an irreducible representation of a semisimple Lie group $G$.
Then there is a parabolic subgroup ${Q} = M_{Q}A_{Q}N_{Q}$ of $G$, a discrete series representation $\sigma$ of $M_{Q}$ and a not necessarily unitary character $\mu: a \mapsto \expe^{\lambda_1\log(a)}$ of $A_{Q}$ such that $\pi$ may be identified with a quotient of the induced representation $\Ind_{Q}^G \sigma \otimes \mu \otimes \iota$; here $\iota$ is the trivial representation of $N_{Q}$.
\end{theorem}

The induced representation is a completion of the left translation representation of $G$ on the space of smooth $\cH_\sigma$-valued functions $f$ on $G$ with the invariance property that
\begin{equation}\label{eq:induced-rep-invariance}
f(xman)
= \sigma(m)^{-1} \expe^{-(\lambda_1-\rho_Q)(\log(a)) } f(x) ;
\end{equation}
here $\rho_Q \in \Lie{a}^*$ is defined by $\rho_Q = \half \sum_{\alpha \in \Sigma^+ \setminus \Sigma(Q)} \dim(\Lie{g}_\alpha) \alpha$.
Note that if $f$ and $f'$ have this property, then $\lip f(km), f'(km) \rip_{\cH_\sigma} = \lip f(k), f'(k) \rip_{\cH_\sigma}$ for all $k \in K$ and all $m \in M_Q$.


This result reduces the classification of irreducible unitary representations of $G$ to the question whether there is a suitable inner product on the space above relative to which the translations act unitarily.
At the time of writing of this paper, there is no complete description of when such an inner product exists.
However, necessary conditions on $\lambda_1$ are known, and in particular there must be an element $w$ of $W$ that normalises $A_Q$ and satisfies $w\Re\lambda_1 = -\Re\lambda_1$ (see \cite[Theorem 16.6]{Kn86}).

The inner product, when it exists, may be written as
\[
\lip f, f' \rip = \int_{K} \lip f(k), \mathcal{C} f'(k)\rip_{\cH_\sigma} \wrt k,
\]
where $\mathcal{C}$ is the operator of convolution with a distribution on $K$.
The representation $\pi$ is a proper quotient of the induced representation if and only if $\mathcal{C}$ has a nontrivial kernel.
This implies that many matrix coefficients $u$ of $\pi$ may be written as an integral of the form
\begin{equation}\label{eq:LK-matrix-coeff}
u(x) = \int_{K} \lip  f(x^{-1} k) , f'(k) \rip_{\cH_\sigma}  \wrt k ,
\end{equation}
where $f$ and $f'$ are functions satisfying the invariance condition above.
The remark after \eqref{eq:induced-rep-invariance} implies that we are effectively dealing with an integral on $K/ K \cap M$.

For irreducible representations, $K$-finite vectors are smooth. 

\subsection*{Zonal spherical functions}

The space $\fnspace{L}^1(G)^\natural$ of integrable $K$-bi-invariant functions on a semisimple Lie group $G$ forms a commutative star-algebra.
The Gel$'$fand space of multiplicative linear functionals on $\fnspace{L}^1(G)^\natural$ may be identified with the set of all bounded spherical functions on $G$, that is,
\[
\int_G f*f'(x) \fn \phi_\lambda(x) \wrt x
= \int_G f(x) \fn \phi_\lambda(x) \wrt x
\times \int_G f'(x') \fn \phi_\lambda(x') \wrt x'
\]
for all $f, f'$ in $\fnspace{L}^1(G)^\natural$ and all bounded spherical functions $\phi_\lambda$.
This formula coupled with some simple changes of variable in $G$ implies that
\begin{equation}\label{eq:functional}
f * \phi_\lambda(x) = f * \phi_\lambda(e) \fn\phi_\lambda(x)
\qquad\text{and}\qquad
\phi_\lambda * f(x) = \phi_\lambda * f(e) \fn\phi_\lambda(x)
\end{equation}
for all $x \in G$; here $e$ denotes the identity element of $G$.

The spherical functions $\phi_\lambda$ are all matrix coefficients of not necessarily unitary representations of $G$, the so-called class one representations, that is, representations which have nontrivial $K$-invariant vectors.

Harish-Chandra's integral formula for the zonal spherical functions (almost) states that
\begin{equation}\label{eq:HC-spherical-fn-int-formula}
\phi_\lambda( g) = \int_{K}   \expe^{-(\lambda -\rho) (H(g^{-1}k))} \wrt k,
\end{equation}
where $x = k(x) \exp(H(x)) n(x)$ is the Iwasawa decomposition of $x \in G$.
From this formula and the asymptotic behaviour of spherical functions (a particular case of the behaviour considered above), it follows that $\phi_\lambda$ is positive-real-valued if and only if $\lambda \in \Lie{a}^*$.
Next, the intrinsic symmetries in semisimple groups mean that $\phi_{\lambda} = \phi_{w\lambda}$ for all $w \in W$.
Thus to parametrise positive-real-valued spherical functions, it suffices to consider $\lambda$ in $((\Lie{a}^*)^+)\afterbar$, the closure of $(\Lie{a}^*)^+$, and ambiguity in the parametrisation arises only if $\lambda$ lies on the boundary of this region.

We shall use spherical functions to measure decay rates, and so it is of interest to be able to compare and estimate spherical functions.
In the next theorem, $\mathrm{conv}(W \mu)$ denotes the convex hull in $\Lie{a}^*$ of the set $\{ w \mu: w \in W\}$, and $\phi_\lambda \leq \phi_\mu$ means that $\phi_\lambda(x) \leq \phi_\mu(x)$ for all $x \in G$.

\begin{theorem}\label{thm:compare-phis}
Suppose that $\lambda, \mu \in \Lie{a}^*$.
Then $\phi_\lambda \leq \phi_\mu$ if and only if $\lambda \in \mathrm{conv}(W \mu)$.
\end{theorem}

\begin{proof}
On the one hand, from the integral formula \eqref{eq:HC-spherical-fn-int-formula}, $|\phi_\lambda| \leq \phi_{\Re(\lambda)}$ for all $\lambda\in \Lie{a}^*_{\C}$.
It then follows that if $\lambda, \lambda' \in \Lie{a}^*$, then
\[
|\phi_{\lambda + z \lambda'}| \leq \max\{ \phi_\lambda, \phi_{\lambda +\lambda'} \} ,
\]
first if $\Re(z)$ is either $0$ or $1$ and then if $0 \leq \Re(z) \leq 1$ by the maximum principle, applied pointwise.
Iterating this interpolation result yields one direction of the theorem.

On the other hand, if $\lambda \notin \mathrm{conv}(W \mu)$, then the formula \eqref{eq:asymptotics} for the asymptotic behaviour of matrix coefficients, applied to the spherical functions, which are matrix coefficients of class one representations of $G$, shows that $\phi_\lambda \not\leq \phi_\mu$.
\end{proof}

\begin{theorem}[{see \cite[Theorem 3.4]{NPP14}}]\label{thm:NPP}
Suppose that $\lambda \in ((\Lie{a}^*)^+)^-$, let $\Sigma_\lambda$ be the set of all positive roots $\alpha$ such that $\lip\lambda, \alpha \rip = 0$, and define the polynomial $p_\lambda$ on $\Lie{a}$ by
\[
p_\lambda(H) =
\prod_{\alpha \in \Sigma_\lambda } (1 +  \alpha(H) )
\qquad\forall H \in \Lie{a}.
\]
Then there are constants $C_1$ and $C_2$, depending on $\lambda$, such that
\[
C_1 \fn p_\lambda(H) \expe^{(\lambda -\rho) (H)}
\leq \phi_\lambda(\exp H)
\leq C_2 \fn p_\lambda(H) \expe^{(\lambda -\rho) (H)}
\qquad\forall H \in (\Lie{a}^+)^-.
\]
\end{theorem}

Finally, we are interested in star-linear functionals on $\fnspace{L}^1(G)^\natural$, that is, in hermitean spherical functions.
Kostant \cite{Ko69} showed that $\phi_\lambda$ is hermitean if and only if there exists $w \in W$ such that $w\lambda = -\bar\lambda$, that is, $w \Re(\lambda) = -\Re(\lambda)$ and $w \Im(\lambda) = \Im(\lambda)$.

As noted above, in some semisimple Lie groups, there exists $w$ such that $w\lambda = -\lambda$ for all $\lambda \in \Lie{a}^*$, and all $\phi_\lambda$ are hermitean when $\lambda$ is real.
In others, the set of real $\lambda$ for which $\phi_\lambda$ is hermitean is a union of subspaces of dimension strictly less than $r$.
However, in all cases $\phi_\rho$ is the constant function $1$, and $\phi_{t\rho}$ is hermitean for all real $t$ and bounded if and only if $|\Re(t)| \leq 1$.
If $\phi_\lambda$ is positive definite and positive-real-valued, then $\phi_\lambda$ is \emph{a fortiori} hermitean, and so in some cases, the set of $\lambda$ for which $\phi_\lambda$ is positive definite is a union of subsets of dimension strictly less than $r$.
As noted above, if $\pi$ is an irreducible unitary representation of $G$, and $\lambda$ is a leading exponent in the asymptotic expansion of the $K$-finite matrix coefficients of $\pi$, as in Theorem \ref{thm:asymptotics}, then $\phi_\lambda$ is a hermitean positive-real-valued spherical function.

Generally, $\Xi$ is used rather than $\phi_0$.
We prefer $\phi_0$, as this notation emphasises that we are dealing with just one of many positive-real-valued spherical functions.

\subsection*{History of growth estimates for matrix coefficients}

If $G$ is an amenable locally compact group, then its Fourier algebra $\fnspace{A}(G)$, that is, the set of matrix coefficients of the regular representation of $G$ on $\fnspace{L}^2(G)$, which may be appropriately normed (see \cite{Ey64}), contains an approximate identity for multiplication.
There are many ways in which we can make this precise; one is to affirm that there exists a net of functions $(u_\alpha: \alpha \in \Alpha)$ such that $\lnorm u_\alpha \rnorm_{\fnspace{A}} \leq 1$ and
$u_\alpha \to 1$ locally uniformly on $G$.
In particular, if $G$ is not compact, then $\lnorm u_\alpha \rnorm_q \to \infty$ for all finite $q$.
There are many results on amenability of this kind; see, e.g., \cite{Gr69}, \cite{Pi84}.

We now suppose that $\pi$ is an irreducible unitary representation of a semisimple Lie group $G$.
Theorem \ref{thm:asymptotics}, coupled with \eqref{eq:Haar-measure}, suggests that the matrix coefficient $\lip \pi(\cdot) \xi, \eta \rip$ belongs to $\fnspace{L}^{q}(G)$ for certain $q \in [2,\infty]$; the asymptotic expression breaks down as $H$ approaches the walls of $\Lie{a}^+$, and so more information is needed to draw this conclusion in the higher rank case, but the conclusion is still correct.
The unsatisfactory aspect of this observation is that the asymptotic expansion does not hold for all matrix coefficients, nor for all unitary representations.
It would be nice to have estimates of the form
\[
\labs \lip \pi(k_1\exp(H)k_2) \xi, \eta \rip \rabs
\leq \lnorm\xi\rnorm_{\Hil_\pi} \lnorm\eta\rnorm_{\Hil_\pi} p(H) \fn\expe^{ (\Re\lambda - \rho)(H)}
\]
for all $H \in \Lie{a}^+$, all $k_1, k_2 \in K$, and all vectors $\xi, \eta \in \calH_\pi$.
Unfortunately, such estimates are impossible---a translate of a matrix coefficient is another matrix coefficient and the new vectors have the same norms as the old vectors, but the ``bump'' where the matrix coefficient is ``large'' can be moved out to infinity, contradicting this decay estimate.
It is possible to give estimates for $K$-finite vectors---indeed, Howe \cite{Ho82} does just this with his concept of $(\Phi, \Psi)$-boundedness.

The Kunze--Stein phenomenon (see \cite{C78}) shows that matrix coefficients of the regular representation belong to $\fnspace{L}^{2+}(G)$, that is, they belong to $\fnspace{L}^{2+\epsilon}(G)$ for all $\epsilon \in \R^+$.
Conversely (see \cite{CHH89}), if sufficiently many matrix coefficients of a unitary representation $\pi$  belong to $\fnspace{L}^{2+}(G)$, then all matrix coefficients do.


To each nontrivial irreducible unitary representation $\pi$ of a semisimple Lie group $G$ on a Hilbert space $\Hil_\pi$, there exist $q \in [2,\infty)$ and a constant $C(\epsilon)$ (both depending on $\pi$) such that
\begin{equation}\label{eq:Lp+-estimate}
\lnorm \lip \pi(\cdot) \xi, \eta \rip \rnorm_{q + \epsilon}
        \leq C(\epsilon) \lnorm\xi\rnorm_{\Hil_\pi} \lnorm\eta\rnorm_{\Hil_\pi}
\qquad\forall \xi, \eta \in \Hil_\pi
\end{equation}
for all positive $\epsilon$ (see \cite{C79}).
We abbreviate this condition to $\lip \pi(\cdot) \xi, \eta \rip  \in \fnspace{L}^{q+}(G)$.
When \eqref{eq:Lp+-estimate} holds only for some $q > 2$, we speak of complementary series representations; such representations do not appear in the Plancherel formula.
In \cite{C78}, it was shown that, if one nonzero matrix coefficient $\lip \pi(\cdot) \xi, \eta \rip$ of an irreducible unitary representation lies in $\fnspace{L}^{q+}(G)$, then all matrix coefficients satisfy the same estimate, but with $q$ replaced by $2N$, where $N$ is an integer such that $q < 2N$.
Various improvements were made to this, such as replacing $q < 2N$ by $q \leq 2N$ (see \cite{CHH89}), and removing the need to increase $q$ for real-rank $1$ groups (see \cite{C83}, which uses a detailed analysis of the representations of these groups).
It is also possible to replace irreducible representations by more general unitary representations as long as we assume that $\lip \pi(\cdot) \xi, \eta \rip  \in \fnspace{L}^{q+}(G)$ for all vectors $\xi$ and $\eta$ in a dense subset of $\Hil_\pi$.
This kind of information has been used in representation theory and its applications (see, for example, \cite{MNS00, BK15, GN15, GGN18}).

Recently, Samei and Wiersma \cite{SW18} found a functional analytic proof that if one matrix coefficient of an irreducible representation satisfies an $\fnspace{L}^{q+}(G)$ estimate for some $q \in [2,+\infty)$, then all matrix coefficients do.
Similar results appear in de Laat and Siebenand \cite{LS21}.
Their arguments involve the construction of a family of ``exotic $C^*$-algebras".
We are going to use a different form of their ideas.
Their elegant argument does not apply to all locally compact groups, but it does to many, including all semisimple Lie groups.
The key is understanding various $C^*$-algebras associated to group representations.
The crucial step is a formula of \cite[Lemma 2.3]{SW18} that generalises a result of \cite[proof of Theorem 1]{CHH89}, as follows.

\begin{lemma}\label{lem:SW}
Suppose that $\pi$ is a unitary representation of a locally compact group $G$ on a Hilbert space $\mathcal{H}_\pi$, and that $\mathcal{H}_0$ is a dense subspace of $\mathcal{H}_\pi$.
Then for all $f \in \fnspace{L}^1(G)$,
\begin{equation}\label{eq:SW}
  \lnorm \pi(f) \rnorm
= \sup_{\xi \in \mathcal{H}_0}  \lim_{n \to \infty}
            \lip  \lpar \pi(f^* * f)^{(*n)}\rpar \xi, \xi \rip^{1/2n}  .
\end{equation}
\end{lemma}

The reader will observe that this is a variant of the spectral radius formula.
The powerful---almost magical---aspect of this is that the taking of $2n$th roots means that the norm of $\xi$ becomes irrelevant; also, if we have estimates for the matrix coefficients of the form
\[
\labs \lip \pi(x) \xi, \xi \rip \rabs \leq C \phi_\lambda(x)
\]
(for a positive constant $C$ and a positive-real-valued hermitean spherical function $\phi_\lambda$), then it will follow that
\[
\labs\lip  \pi\lpar (f^* * f)^{(*n)}\rpar \xi, \xi \rip \rabs^{1/2n}
\leq \lpar C \int_G (f^* * f)^{(*n)} \phi_\lambda(x) \wrt x \rpar^{1/2n} ,
\]
and the constant $C$ will become irrelevant as $n \to \infty$.

Unfortunately, $\fnspace{L}^{q+}(G)$ estimates are not very precise, especially when the real-rank $r$ of $G$ is greater than $1$.
Figure \ref{fig:Kn-Sp} (due to Knapp and Speh \cite{KS83}) shows the $\lambda \in ((\Lie{a}^*)^+)\afterbar$ that parametrise the positive definite positive-real-valued spherical functions $\phi_\lambda$ in the case in which $G$ is $\group{SU}(8,2)$.
The behaviour of these spherical functions at infinity involves $\lambda$ playing the role of $\lambda_1$ in Theorem \ref{thm:asymptotics}.
It is clear that on this group there are positive definite spherical functions that belong to the same $\fnspace{L}^{q+}(G)$ spaces but have very different asymptotic behaviour.

\begin{figure}
\begin{center}
\begin{tikzpicture}[scale=0.5]
\draw[dashed] (0,0) -- (0,10);
\draw [dashed](0,0) -- (10,10);
\draw (0,3) -- (4,7);
\draw (0,5) -- (2,7);

\draw[color=red] (0,6.75) -- (5.25,6.75) -- (6,6);
\path (6.5,5.2) node [color=red]{$p=6$}  ;

\filldraw (0,0) -- (0,3) -- (1.5,1.5) -- cycle ;
\filldraw (0,3) -- (0,5) -- (1,4) -- cycle ;
\filldraw (0,5) -- (0,7) -- (1,6) -- cycle ;
\path (6,10) node {$(\Lie{a}^*)^+$} ;
\end{tikzpicture}
\end{center}
\caption{Parameters for the class one complementary series}
\end{figure}\label{fig:Kn-Sp}

Another unsatisfactory aspect of $\fnspace{L}^{q+}(G)$ estimates is that they contain no pointwise information.

%
%
%
%

\subsection*{Other notation}
Expressions like $c$ and $C(\xi,\eta)$ denote \emph{constants} that  may vary from one instance to another: these are positive numbers that may depend on the ambient group, or the representation $\pi$, but not on any specifically quantified parameters.

\section{Proofs}

The proofs are mostly functional analytic.
We write $\fnspace{C}_c(G)$ and $\fnspace{L}^p(G)$ for the usual space of compactly supported continuous functions and the standard Lebesgue space on $G$.

\subsection*{Proof of Theorem A}
We restate Theorem A for the reader's convenience.

\begin{theoremA}
Let $\pi$ be a unitary representation of a semisimple Lie group $G$, and suppose that $\lambda \in (\Lie{a}^{+})\afterbar$.
Consider the condition
\begin{equation}\label{eq:strong-assumption}
\mathcal{A} \lip \pi(\cdot) \xi, \eta\rip
\leq \lnorm \xi \rnorm_{\Hil_\pi} \lnorm \eta \rnorm_{\Hil_\pi}
    \phi_\lambda
\qquad\forall \xi, \eta \in \cH_\pi.
\end{equation}
Then \eqref{eq:strong-assumption} holds if and only if all $\sigma \in \bar{G}$ such that $\sigma \preceq \pi$ satisfy \eqref{eq:strong-assumption} with $\pi$ replaced by $\sigma$.
\end{theoremA}

\begin{proof}
Suppose that $\sigma \in \bar{G}$ and that $\sigma \preceq \pi$.
For all $u \in \fnspace{A}_\sigma$, there exists a net $v_i$ in $\fnspace{A}_\pi$ such that
$\norm{v_i}_{\fnspace{B}} \leq \norm{u}_{\fnspace{B}}$ and $v_i \to u$ locally uniformly on $G$.
It follows immediately that $\mathcal{A} v_i \to \mathcal{A} u$ locally uniformly on $A$, and hence
\[
\mathcal{A} u \leq \norm{u}_{\fnspace{B}} \phi_{\lambda}
\qquad\forall u \in \fnspace{A}_\pi(G).
\]
This implies that \eqref{eq:strong-assumption} holds with $\sigma$ in place of $\pi$.

Conversely, since $\pi$ is weakly equivalent to the sum of all $\sigma \in \hat G$ such that $\sigma \preceq \pi$, part (1) of the theorem follows by a similar argument.
\end{proof}

\subsection*{Construction of an exotic algebra}

\begin{lemma}\label{lem:A-conv}
Suppose that $f, f' \in \fnspace{C}_c(G)$.
Then
\[
\mathcal{A} (f*f')
\leq \mathcal{A} (f) * \mathcal{A} (f')   .
\]
Hence, if $\phi_\lambda$ is a positive-real-valued spherical function, then
\[
\int_G \mathcal{A} (f*f')(x) \fn \phi_\lambda(x) \wrt x
\leq \int_G \mathcal{A} (f)(y) \fn \phi_\lambda(y) \wrt y
\times \int_G \mathcal{A} (f')(x) \fn \phi_\lambda(x) \wrt x  .
\]
\end{lemma}

\begin{proof}
This is an exercise in integration theory.
Indeed, for $k, k', k'' \in K$,
\begin{align*}
|f*f'(k'xk'')|
&\leq \int_G |f(y)| \fn |f'(y^{-1}k'xk'')| \wrt y  \\
&= \int_G |f(k'yk)| \fn |f'(k^{-1} y^{-1}xk'')| \wrt y  ,
\end{align*}
whence
\begin{align*}
|f*f'(k'xk'')|
&\leq \int_K \int_G |f(k'yk)| \fn |f'(k^{-1}y^{-1}xk'')| \wrt y \wrt k  \\
&= \int_G \int_K |f(k'yk)| \fn |f'(k^{-1}y^{-1}xk'')| \wrt k \wrt y  \\
&\leq \int_G \lpar \int_K |f(k'yk)|^2 \wrt k \rpar^{1/2} \lpar\int_K |f'(k^{-1}y^{-1}xk'')|^2 \wrt k\rpar^{1/2} \wrt y.
\end{align*}
We deduce that
\begin{align*}
\mathcal{A} (f&*f')(x) \\
&= \lpar\int_K \int_K |f*f'(k'xk'')|^2 \wrt k' \wrt k'' \rpar^{1/2} \\
&\leq \bigglpar\int_K \int_K \bigglabs \int_G \lpar \int_K |f(k'yk)|^2 \wrt k \rpar^{1/2}
\lpar\int_K |f'(k^{-1}y^{-1}xk'')|^2 \wrt k\rpar^{1/2} \wrt y \biggrabs^2 \wrt k' \wrt k'' \biggrpar^{1/2} \\
&\leq \int_G \bigglpar\int_K \int_K \bigglabs \lpar \int_K |f(k'yk)|^2 \wrt k \rpar
\lpar\int_K |f'(k^{-1}y^{-1}xk'')|^2 \wrt k\rpar \biggrabs \wrt k' \wrt k'' \biggrpar^{1/2} \wrt y  \\
&= \int_G \lpar\int_K \int_K |f(k'yk)|^2 \wrt k \wrt k'\rpar^{1/2}
\lpar\int_K \int_K |f'(k^{-1}y^{-1}xk'')|^2 \wrt k \wrt k''\rpar^{1/2}  \wrt y  \\
& = \mathcal{A} (f) * \mathcal{A} (f') (x)        .
\end{align*}

The second part of the lemma follows immediately.
\end{proof}

For $f \in \fnspace{C}_c(G)$, we define
\begin{equation}\label{eq:def-norm-lambda}
\norm{f}_{(\lambda)} = \int_{G} \mathcal{A}(f)(x) \fn \phi_{\lambda}(x) \wrt x.
\end{equation}

\begin{lemma}\label{lem:C-star-alg}
Suppose that $\lambda \in (\Lie{a}^+)\afterbar$ and the spherical function $\phi_\lambda$ is bounded.
Then the norm $\norm{\cdot}_{(\lambda)}$ is a Banach algebra norm on $\fnspace{C}_c(G)$, that is,
\begin{enumerate}
  \item[(a)] $\norm{f}_{(\lambda)} \geq 0$ and equality holds if and only if $f = 0$
  \item[(b)] $\norm{f+ g}_{(\lambda)} \leq \norm{f}_{(\lambda)}  +\norm{g}_{(\lambda)}$
  \item[(c)] $\norm{cf}_{(\lambda)} = |c| \norm{f}_{(\lambda)}$
  \item[(d)] $\norm{f * g}_{(\lambda)} \leq \norm{f}_{(\lambda)} \norm{g}_{(\lambda)}$
\end{enumerate}
for all $f, g \in \fnspace{C}_c(G)$ and all $c \in\C$.
Further, if $\lambda \in \Lie{a}^{*,\Her}$, then
\begin{enumerate}
\item[(e)]
$\norm{f} = \norm{f^*}$ for all $f \in \fnspace{C}_c(G)$,
\end{enumerate}
and $\norm{\cdot}_{(\lambda)}$ is a star-algebra norm.
\end{lemma}
\begin{proof}
Items (a) to (c) are obvious, and (d) follows from Lemma 4.1.
Finally, if $\phi_\lambda$ is hermitean, then
\begin{align*}
\int_{G} \mathcal{A}(f^*)(x) \fn \phi_{\lambda}(x) \wrt x
= \int_{G} \mathcal{A}(f)(x^{-1}) \fn \phi_{\lambda}(x) \wrt x
= \int_{G} \mathcal{A}(f)(x) \fn \phi_{\lambda}(x) \wrt x
\end{align*}
for all $f \in \fnspace{C}_c(G)$, as claimed.
\end{proof}

We define $\fnspace{E}_{(\lambda)}(G)$ to be the convolution algebra $\fnspace{C}_c(G)$, equipped with the norm  $\norm{\cdot}_{(\lambda)}$.
Note that there is no reason to suppose that $\norm{f*f^*}_{(\lambda)} = \norm{f}_{(\lambda)}^2$, nor does the star-algebra $\fnspace{E}_{(\lambda)}(G)$ have a bounded approximate identify.
We wish to construct the enveloping $C^*$-algebra of $\fnspace{E}_{(\lambda)}(G)$, and first we examine the dual space of $\fnspace{E}_{(\lambda)}(G)$.

\begin{lemma}\label{lem:duality}
Suppose that $u$ is a continuous function on $G$ and $\lambda \in (\Lie{a}^+)\afterbar$.
Then the following are equivalent:
\begin{enumerate}
  \item[(a)] $\ds\labs\int_{G} u(x) \fn f(x) \wrt x \rabs \leq C \norm{f}_{(\lambda)}$  for all $f \in \fnspace{C}_c(G)$;
  \item[(b)] $\ds \mathcal{A} u \leq C \phi_\lambda$.
\end{enumerate}
\end{lemma}

\begin{proof}
This is an exercise in standard inequalities in integration theory.
\end{proof}

The next step is to show that $\fnspace{E}_{(\lambda)}(G)$ has a nondegenerate representation on a Hilbert space if $\lambda \in (\Lie{a}^{+})\afterbar$.

\begin{lemma}\label{lem:dual-of-A-lambda}
If $\lambda \in (\Lie{a}^{+})\afterbar$ and $f \in \fnspace{E}_{(\lambda)}(G)$, then for all $g \in \fnspace{L}^2(G)$, $f*g \in \fnspace{L}^2(G)$ and
\[
\norm{f * g}_{\fnspace{L}^2(G)}
\leq \norm{f}_{(\lambda)} \norm{g}_{\fnspace{L}^2(G)}.
\]
Further, if $f*g = 0$ in $\fnspace{L}^2(G)$ for all $g \in \fnspace{L}^2(G)$, then $f = 0$ in $\fnspace{E}_{(\lambda)}(G)$.
\end{lemma}

\begin{proof}
First, $\phi_0 \leq \phi_\lambda$ by Theorem \ref{thm:compare-phis}.
From \cite{CHH89}, it is known that for all $u \in \fnspace{A}(G)$,
\[
\mathcal{A} u \leq \norm{u}_\fnspace{B} \fn\phi_0.
\]
It follows that
\begin{align*}
\labs \int_{G} u(x) \fn f(x) \wrt x \rabs
&\leq \int_{G} \mathcal{A} (u)(x) \fn \mathcal{A} (f)(x) \wrt x\\
&\leq \norm{u}_\fnspace{B}  \int_{G} \phi_0(x) \fn \mathcal{A} (f)(x) \wrt x \\
&= \norm{u}_\fnspace{B} \norm{f}_{(0)} \\
&\leq  \norm{u}_\fnspace{B} \norm{f}_{(\lambda)}
\end{align*}
for all $f \in \fnspace{E}_{(\lambda)}(G)$.
The duality between $\fnspace{A}(G)$ and the reduced $C^*$-algebra of $G$ implies that there is a norm-nonincreasing homomorphism of $\fnspace{E}_{(\lambda)}(G)$ into the said $C^*$-algebra.

The last assertion is easy to check, so the homomorphism is an embedding.
\end{proof}

This representation is a star-representation if $\lambda \in \Lie{a}^{*,\Her}$.

\begin{lemma}\label{lem:B-lambda-fns}
Suppose that $\lambda \in \Lie{a}^{*,\Her}$ and $\pi \in \bar{G}$.
Then the following are equivalent:
\begin{enumerate}
\item[(a)] for all $\xi$ and $\eta$ in a dense subspace $\cH^0_\pi$ of $\cH_\pi$, there is a constant $ C(\xi,\eta)$ such that
\[
\labs\int_G f(x) \lip\pi(x) \xi,\eta\rip \wrt x \rabs
\leq C(\xi,\eta) \norm{f}_{(\lambda)}
\qquad\forall f \in \fnspace{C}_c(G);
\]

\item[(b)] for all $\xi$ and $\eta$ in $\cH_\pi$,
\[
\labs\int_G f(x) \lip\pi(x) \xi,\eta\rip \wrt x \rabs
\leq  \norm{\xi}_{\cH_\pi} \norm{\eta}_{\cH_\pi}  \norm{f}_{(\lambda)}
\quad\forall f \in \fnspace{C}_c(G);
\]

\item[(c)] for all $\xi$ and $\eta$ in a dense subspace $\cH^0_\pi$ of $\cH_\pi$, there is a constant $ C(\xi,\eta)$ such that
\[
\mathcal{A}(\lip\pi(\cdot) \xi,\eta\rip) \leq C(\xi,\eta) \fn \phi_\lambda ;
\]
\item[(d)] for all $\xi$ and $\eta$ in a dense subspace $\cH^1_\pi$ of $\cH_\pi$, containing only $K$-finite vectors, there is a constant $ C(\xi,\eta)$ such that
\[
\labs\lip\pi(\cdot) \xi,\eta\rip\rabs
\leq C(\xi,\eta) \fn \phi_\lambda ;
\]

\item[(e)] for all smooth vectors $\xi$ and $\eta$ in $\cH_\pi$, there is a constant $ C(\xi,\eta)$ such that
\[
\labs\lip\pi(\cdot) \xi,\eta\rip\rabs
\leq C(\xi,\eta) \fn \phi_\lambda  ;
\]

\item[(f)] for all $\xi$ and $\eta$ in $\cH_\pi$,
\[
\mathcal{A}(\lip\pi(\cdot) \xi,\eta\rip)
\leq   \norm{\xi}_{\cH_\pi} \norm{\eta}_{\cH_\pi}  \fn \phi_\lambda;
\]

\item[(g)] for all $u$ in $\fnspace{A}_\pi(G)$,
\[
\mathcal{A}(u)
\leq   \norm{u}_{\fnspace{B}} \fn \phi_\lambda.
\]
\end{enumerate}
\end{lemma}

\begin{proof}
Lemma \ref{lem:duality} shows that (a) and (c) are equivalent and that (b) and (f) are equivalent; moreover, (f) and (g) are equivalent by definition of $\fnspace{A}_\pi(G)$.
It is easy to see that (c), with $\cH_\pi^0$ replaced by $\cH_\pi^1$, and (d) are equivalent, and obvious that (e) implies (d) and that both (b) and (e) imply (c).
It remains therefore to prove that (c) implies (b) and (f) implies (e).

Suppose that (c) holds, and take $\xi \in \cH_\pi^0$. 
Then
\begin{equation}\label{eq:pre-C-star}
\labs \int_G f(x) \lip\pi(x) \xi,\xi\rip \wrt x \rabs
\leq C(\xi) \norm{f}_{(\lambda)}
\qquad\forall f \in \fnspace{C}_c(G).
\end{equation}
It follows that the Gel$'$fand--Na{\u\i}mark--Segal (GNS) representation $\sigma_\xi$ associated to $\lip\pi(\cdot) \xi,\xi\rip$ is bounded on the star algebra $\fnspace{E}_{(\lambda)}(G)$.
From Lemma \ref{lem:SW}, coupled with parts (d) and (e) of  Lemma \ref{lem:C-star-alg}, we deduce that
\[
\norm{\sigma_\xi(f)}_{\cH_{\sigma_\xi}} \leq \norm{f}_{(\lambda)}.
\]
Since $\pi$ is weakly equivalent to the direct sum of all GNS representations $\sigma_\xi$ as $\xi$ runs over $\cH_\pi^0$, it follows that
\[
\norm{\pi(f)}_{\cH_{\pi}} \leq \norm{f}_{(\lambda)},
\]
which implies (b).

Finally, if (f) holds, and $\xi$ and $\eta$ are smooth vectors, then we may decompose $\xi$ and $\eta$ as in \eqref{eq:xi-decomp}, and argue as in the Proposition of \cite[p.~105]{CHH89}; we deduce that
\begin{align*}
\labs\lip\pi(\cdot) \xi,\eta\rip\rabs
&\leq \sum_{\tau,\tau' \in \hat{K}}
        \labs\lip\pi(\cdot) \xi_\tau,\eta_{\tau'}\rip\rabs  \\
&\leq \sum_{\tau,\tau' \in \hat{K}}
        \dim(\cH_\tau)^{1/2} \dim(\cH_{\tau'})^{1/2}
        \mathcal{A}(\lip\pi(\cdot) \xi_\tau,\eta_{\tau'}\rip)  \\
&\leq \sum_{\tau,\tau' \in \hat{K}}
        \dim(\cH_\tau)^{1/2} \norm{\xi_\tau}_{\cH_\pi} \dim(\cH_{\tau'})^{1/2}
       \norm{\eta_{\tau'}}_{\cH_\pi} \fn \phi_\lambda  \\
&\leq C(\xi,\eta) \fn \phi_\lambda ,
\end{align*}
from \eqref{eq:Weyl-dim-formula} and \eqref{eq:smooth-decomp}, so (e) holds.
\end{proof}

Condition (e) could be improved to requiring that $\xi$ and $\eta$ lie in a Sobolev space; the order of differentiability needed depends on the group $G$.
The main point of (e), however, is that the space of smooth vectors is $G$-invariant while the space of $K$-finite vectors is not.

\begin{definition}
For all $\lambda \in \Lie{a}^{*,\Her}$, we define $\fnspace{C}^*_{(\lambda)}(G)$ to be the enveloping $C^*$-algebra of $\fnspace{E}_{(\lambda)}(G)$; we take $\fnspace{B}_{(\lambda)}(G)$ to be its dual space, and $\bar{G}_{(\lambda)}$ to be the collection of all unitary representations $\pi$ of $G$ such that
any of the equivalent conditions of Lemma \ref{lem:B-lambda-fns} hold.
\end{definition}

From the definitions, $\fnspace{C}^*_{(\lambda)}(G)$ is the completion of $\fnspace{C}_c(G)$ in the norm
\[
\norm{f}_{(\lambda)}= \sup\{ \norm{\pi(f)} : \pi \in \bar{G}_{(\lambda)}\},
\]
while $\fnspace{B}_{(\lambda)}(G)$ is the space of matrix coefficients of the unitary representations in $\bar{G}_{(\lambda)}$, and is a weak-star topology closed subspace of $\fnspace{B}(G)$.

The algebra  $\fnspace{C}^*_{(\lambda)}(G)$ is an ``exotic $C^*$-algebra'', in the sense that it lies ``between'' the reduced $C^*$-algebra of $G$ and the full $C^*$-algebra of $G$.
It is also exotic in the sense that it arises as the completion of a star-algebra that does not contain an approximate identity.

\subsection*{Proof of Theorem B}

We now reproduce and prove Theorem B.

\begin{theoremB}
Let $\pi$ be a unitary representation of a semisimple Lie group $G$, and suppose that $\lambda \in \Lie{a}^{*,\Her}$.
Then the following are equivalent:
\begin{equation}\label{eq:cond-1}
\mathcal{A} \lip \pi(x) \xi, \eta\rip
\leq \lnorm \xi \rnorm_{\Hil_\pi} \lnorm \eta \rnorm_{\Hil_\pi}
    \phi_\lambda(x)
\qquad\forall x \in G
\end{equation}
for all $\xi$ and $\eta$ in $\cH_\pi$, and
\begin{equation}\label{eq:cond-2}
\sup_{k,k' \in K} \labs \lip \pi (k x k') \xi, \eta\rip\rabs  \leq C(\xi, \eta) \fn\phi_\lambda(x)
\qquad\forall x \in G
\end{equation}
for all $k,k' \in K$ and for all $K$-finite $\xi, \eta \in \cH_\pi$.
Further, if these conditions hold, then we may take $C(\xi, \eta) $ to be given by
\[
C(\xi, \eta) = \dim(\Span(\pi(K)\xi))^{1/2} \lnorm \xi \rnorm_{\Hil_\pi} \dim(\Span(\pi(K)\eta))^{1/2}  \lnorm \eta \rnorm_{\Hil_\pi},
\]
and moreover, \eqref{eq:cond-2} holds for all smooth $\xi, \eta \in \cH_\pi$.

If moreover $\pi$ is irreducible, then conditions \eqref{eq:cond-1} and \eqref{eq:cond-2} both hold provided that there exist $\xi$ and $\eta$ in $\cH_\pi \setminus \{0\}$
such that
\begin{equation}\label{eq:cond-3}
\mathcal{A} \lip \pi(x) \xi, \eta\rip
\leq C(\xi,\eta)  \fn\phi_\lambda(x)
\qquad\forall x \in G.
\end{equation}
\end{theoremB}

\begin{proof}
Lemma \ref{lem:B-lambda-fns} shows that \eqref{eq:cond-1} and \eqref{eq:cond-2} are equivalent.
The description of the constant $C(\xi,\eta)$  is part of one of the main theorems of \cite{CHH89}, where an estimate in terms of the $K$-types that appear is also proved, for the particular case of a tempered representation; there are no problems in extending the argument.

Suppose that \eqref{eq:cond-3} holds.
Lemma \ref{lem:A-conv} and formula \eqref{eq:functional} imply that, for all $f, f' \in \fnspace{C}_c(G)$,
\[
\begin{aligned}
\mathcal{A} (\lip \pi(\cdot) \pi(f') \xi, \pi(f^*) \eta\rip )
&= \mathcal{A} (f' * \lip \pi(\cdot) \xi, \eta\rip * f) \\
&\leq \mathcal{A} (f') * (\mathcal{A}\lip \pi(\cdot) \xi, \eta\rip) * \mathcal{A}(f) \\
&\leq C(\xi,\eta) \fn\mathcal{A} (f') * \phi_\lambda * \mathcal{A}(f) \\
&= C(f,f',\xi,\eta) \fn\phi_\lambda.
\end{aligned}
\]
Since $\{ \pi(f') \xi : f' \in \fnspace{C}_c(G)\}$ and $\{ \pi(f) \eta : f \in \fnspace{C}_c(G)\}$ are dense subspaces of $\cH_\pi$ that contain the space $\cH_\pi^K$ of $K$-finite vectors, another application of Lemma \ref{lem:B-lambda-fns} completes the proof.
\end{proof}

\section{Applications}

These ideas give new information on the representation theory of semisimple groups.

\subsection*{Irreducible unitary representations}
There is a connection between the parameters of the Langlands classification of a unitary representation and the spherical function that controls the decay of the asymptotic expansion.

\begin{theorem}\label{thm:irred-case}
Suppose that $\pi \in \hat{G}$, and that $\lambda$ is a leading exponent of the asymptotic expansion of the $K$-finite matrix coefficients of $\pi$, as in Theorem \ref{thm:asymptotics}.
Then
\[
\mathcal{A}(u) \leq \norm{u}_{\fnspace{A}_\pi} \fn\phi_{\Re(\lambda)}
\qquad\forall u \in \fnspace{A}_\pi(G).
\]
\end{theorem}

\begin{proof}
By Theorem B, it will suffice to show that a nontrivial $K$-finite matrix coefficient of $\pi$ satisfies an estimate of the desired form, but without control of the factor multiplying the spherical function.

By  Theorem \ref{thm:Langlands-Knapp-clsfctn}, the matrix coefficients of $\pi$ are matrix coefficients of an induced representation $\Ind_{Q}^G \sigma \otimes \mu \otimes \iota$, where  $\mu$ is hermitean and $\sigma$ is a discrete series representation of $M_Q$.
We may identify $\sigma$ with a subrepresentation of the left regular representation of $M_Q$.
It follows from \eqref{eq:induced-rep-invariance}, \eqref{eq:LK-matrix-coeff} and this identification that a nonzero $K$-finite matrix coefficient $u$ may be written in the form
\begin{equation}\label{eq:LK-matrix-coeff-1}
u(x) = \int_{K} \int_{\MQ} f(x^{-1} k', m)  \bar f'(k', m)   \wrt m \wrt k' ,
\end{equation}
where $f$ and $f'$ are $K$-finite functions on $G$ with values in $\fnspace{L}^2(\MQ)$ that satisfy the condition
\begin{equation}\label{eq:invariance-prop-1}
f(xman, m')
= \expe^{-(\lambda-\rho_Q)(\log(a)) } f(x, mm')
\end{equation}
for all $x \in G$, all $m,m' \in M_Q$, all $a \in A_Q$ and all $n \in n_Q$.

Now the group $\MQ$ has an Iwasawa decomposition $(K \cap \MQ) (A \cap \MQ) (N \cap \MQ)$, and the Haar measure on $\MQ$ may be written
\[
\begin{aligned}
\int_{\MQ} h(m) \wrt m
&= \int_{K \cap \MQ}\int_{N \cap \MQ} \int_{A \cap \MQ} h(kna) \wrt a \wrt n \wrt k \\
&= \int_{K \cap \MQ}\int_{A \cap \MQ} \int_{N \cap \MQ} \expe^{2\rho_M(\log(a))}
h(kan) \wrt n \wrt a \wrt k,
\end{aligned}
\]
where $\rho_{M} = \half \sum_{\alpha \in \Sigma^+ \cap \Sigma(Q)} \dim(\Lie{g}_\alpha) \alpha$.
The Iwasawa decomposition for $\MQ$ means that we may write $m \in \MQ$ in the form $k_Q a_Q n_Q$, where $k_Q \in {K \cap \MQ}$, $a_Q \in A \cap \MQ$ and $n_Q \in N \cap \MQ$.

The assumptions on $f$ imply that
\[
\lpar \int_{\MQ} \labs f(k', m) \rabs^2 \wrt m \rpar^{1/2}
\]
is uniformly bounded as $k'$ varies over $K$.
Since $f(k'k, m) = f(k',km)$ for all $k$ in $K$, $k' \in K \cap \MQ$ and all $m \in \MQ$, the $\fnspace{L}^2(\MQ)$-function $f(k' ,\cdot)$ is $K \cap \MQ$-finite under left translations, and the finitely many irreducible representations of $K \cap \MQ$ that appear are independent of $k \in K$.
It follows from the $K$-finiteness of the $\fnspace{L}^2(\MQ)$-valued function $f$ on $K$ that
\begin{equation}\label{eq:quasi-reg-est}
\lpar \int_{A \cap \MQ}\int_{N\cap \MQ} \labs f(k', a n) \rabs^2 \expe^{2\rho_M(\log(a))} \wrt n \wrt a \rpar^{1/2} \leq C,
\end{equation}
uniformly for $k' \in K$.
The same holds for $f'$, with a constant $C'$.

From the formula \eqref{eq:LK-matrix-coeff-1} for the matrix coefficient $u$, the Iwasawa decomposition of $M_Q$, the invariance property \eqref{eq:invariance-prop-1}, the right invariance of Haar measure on $K$,  H\"older's inequality, and \eqref{eq:quasi-reg-est} for $f'$, we see that
\begin{align*}
|u(x)|
&= \labs \int_{K} \int_{\MQ} f(x^{-1} k', m)  \bar f'(k', m)   \wrt m \wrt k' \rabs\\
&\leq \int_{K} \int_{K \cap \MQ} \int_{A \cap \MQ} \int_{N \cap \MQ}
\labs f(x^{-1}k' , kan) \rabs \labs f'(k', kan) \rabs  \expe^{2\rho_M(\log(a))}  \wrt n \wrt a \wrt k \wrt k' \\
&= \int_{K} \int_{K \cap \MQ} \int_{A \cap \MQ} \int_{N \cap \MQ}
\labs f(x^{-1}k'k , an) \rabs \labs f'(k'k, an) \rabs  \expe^{2\rho_M(\log(a))}  \wrt n \wrt a \wrt k \wrt k' \\
&= \int_{K} \int_{A \cap \MQ} \int_{N \cap \MQ}
\labs f(x^{-1}k' , an) \rabs \labs f'(k', an) \rabs  \expe^{2\rho_M(\log(a))}  \wrt n \wrt a \wrt k' \\
&\leq \int_{K}  \lpar \int_{A \cap \MQ} \int_{N \cap \MQ}
\labs f(x^{-1}k', an)\rabs^2  \expe^{2\rho_M(\log(a))}   \wrt n \wrt a \rpar^{1/2} \\
&\hspace{3cm} \times \lpar \int_{A \cap \MQ} \int_{N \cap \MQ}
\labs  f'(k', an) \rabs^2  \expe^{2\rho_M(\log(a))} \wrt n \wrt a \rpar^{1/2}  \wrt k' \\
&\leq C' \int_{K}  \lpar \int_{A \cap \MQ} \int_{N \cap \MQ}
\labs f(x^{-1}k', an)\rabs^2  \expe^{2\rho_M(\log(a))}   \wrt n \wrt a \rpar^{1/2}  \wrt k' .
\end{align*}

We write
\[
x^{-1}k'
= \tilde k \tilde m \tilde a \tilde n
= \tilde k k_Q a_Q n_Q \tilde a \tilde n
\]
for some choice of $\tilde k \in K$, $\tilde m \in \MQ$, $\tilde a \in A_Q$ and $\tilde n \in N_Q$, and then $k_Q \in K \cap \MQ$, $a_Q \in A\cap\MQ$ and $n_Q \in N \cap \MQ$.
Since $A_Q$ and $\MQ$ commute,
\[
x^{-1}k' = (\tilde k k_Q) (a_Q \tilde a) (n_Q \tilde n),
\]
and $\tilde k k_Q  \in  K$, $a_Q \tilde a \in A$ and $n_Q \tilde n \in N$.
In particular, $a_Q \tilde a$ is the $A$-component of the Iwasawa decomposition in $G$ of $x^{-1}k'$.

From the invariance formula \eqref{eq:invariance-prop-1} and integration arguments,
\begin{equation*}
\begin{aligned}
&\int_{A \cap \MQ} \int_{N \cap \MQ}
\labs f(x^{-1}k', an)\rabs^2  \expe^{2\rho_M(\log(a))}   \wrt n \wrt a  \\
&   =  \int_{A \cap \MQ} \int_{N \cap \MQ}
\labs f(\tilde k k_Q a_Q n_Q \tilde a \tilde n, an)\rabs^2  \expe^{2\rho_M(\log(a))}   \wrt n \wrt a  \\
&   =\int_{A \cap \MQ} \int_{N \cap \MQ}
\labs  \expe^{-(\lambda-\rho_Q)(\log(\tilde a)) }f(\tilde k k_Q  , a_Q n_Q an)\rabs^2  \expe^{2\rho_M(\log(a))}   \wrt n \wrt a  \\
&   =  \expe^{-2(\Re(\lambda)-\rho_Q)(\log(\tilde a)) } \expe^{2\rho_M(\log(a_Q))}  \\
&\hspace{3cm} \int_{A \cap \MQ} \int_{N \cap \MQ}
\labs  f(\tilde k k_Q , a_Qn_Q an)\rabs^2   \expe^{2\rho_M(\log(a_Qa))}  \wrt n \wrt a \\
&\leq C^2  \expe^{-2(\Re(\lambda)-\rho_Q)(\log(\tilde a)) } \expe^{2\rho_M(\log(a_Q))} .
\end{aligned}
\end{equation*}
It may be checked that $\rho_Q(\log(\tilde a))+\rho_M(\log(a_Q))  = \rho(\log(\tilde a a_Q) $ and $\lambda(\log(a_Q) = 0$, so
\begin{equation*}
\begin{aligned}
|u(x)|
& \leq C'C \int_{K} \expe^{-(\Re(\lambda)-\rho_Q)(\log(\tilde a)) } \expe^{\rho_M(\log(a_Q))}   \wrt k' \\
& = C'C \int_{K} \expe^{-(\Re(\lambda)-\rho)(\log(\tilde a a_Q)) }   \wrt k' \\
& =   C'C \fn \phi_{\Re(\lambda)}(x),
\end{aligned}
\end{equation*}
as required.
\end{proof}

This in turn allows us to describe the decay of all unitary representations.

\begin{corollary}\label{cor:minimal}
Suppose that $\pi \in \bar{G}$.
Then there exists a unique minimal $\phi_{\lambda^*}$, where $\lambda^* \in ((\Lie{a}^*)^+)\afterbar$, such that
\begin{equation}\label{eq:Glambda}
\mathcal{A} \lip \pi(\cdot) \xi, \eta\rip
\leq \lnorm \xi \rnorm_{\Hil_\pi} \lnorm \eta \rnorm_{\Hil_\pi}
    \phi_{\lambda^*}
\qquad\forall \xi, \eta \in \cH_\pi.
\end{equation}
\end{corollary}

\begin{proof}
We write $\pi$ as a direct integral of irreducible representations, and apply the previous corollary to each irreducible component representation.
Define the subset $\Lambda$ of $\Lie{a}^{*,\Her}$ to be the set of $\Re(\lambda)$ as $\lambda$ runs over the leading exponents in the asymptotic expansions of the matrix coefficients of the irreducible representations that are weakly contained in $\pi$.

By fibering $\hat{G}$ according to the $\Re(\lambda)$ that arise in this way, we deduce that
\begin{equation}\label{eq:fibered}
\mathcal{A} \lip \pi(\cdot) \xi, \eta\rip
\leq \int_{\Lambda} \phi_\lambda \wrt \mu(\lambda),
\end{equation}
where $\mu$ is a positive measure on $\Lambda$, and
\[
\int_{\Lambda} \wrt \mu(\lambda)
\leq \norm{\xi}_{\cH_\pi} \norm{\eta}_{\cH_\pi} .
\]
For each $\lambda \in \Lambda$, it is possible to find matrix coefficients of $\pi$ such that the associated measure $\mu$ is supported in arbitrarily small neighbourhoods of $\Lambda$, and so the integral in \eqref{eq:fibered} may be made arbitrarily close to a multiple of $\fn\phi_\lambda$ (in the topology of locally uniform convergence).

To conclude, we let $\{H_1, \dots, H_r\}$ be a basis of $\Lie{a}$ dual to $\{\alpha_1, \dots, \alpha_r\}$, that is, $\lip \alpha_i, H_j \rip = \delta_{ij}$ for all $\alpha_j \in \Delta$ and all $i \in \{1, \dots, r\}$ (where $\delta_{ij}$ denotes the Kronecker delta).
Now $\phi_{\lambda^*}$ controls the decay of the matrix coefficients of $\pi$ if and only if $\lambda(H_j) \geq \lambda(H_j)$ for all $\lambda \in \Lambda$. We therefore choose $\lambda^*$ such that
\[
\lambda^*(H_j) = \sup\{ \lambda_\sigma(H_j) : \lambda \in \Lambda \},
\]
and this $\lambda^*$ has the desired properties, by Theorem
\ref{thm:compare-phis}.
\end{proof}

Note that if the Weyl group $W$ does not contain $-I$, then $\phi_{\lambda^*}$ need not be hermitean.
We decompose $\pi$ because we need hermitean spherical functions to apply Theorem B.

\subsection*{$\fnspace{L}^{q+}(G)$ representations}
These representations have been considered extensively.
It has been long known that all matrix coefficients of all unitary representations of real-rank one groups lie in $\fnspace{L}^{q+}(G)$ when the $K$-finite matrix coefficients do \cite{C83}; the arguments there do not extend easily to the general case.

Recently Samei and Wiersma \cite{SW18} extended this to general semisimple groups, which served as a source of inspiration for this paper.
Essentially, they proved results similar to those in Section 3 of this paper, with ``belonging to $\fnspace{L}^q(G)$'' instead of ``domination of the root-mean-square average by a spherical function''.
The next corollary connects our results and theirs.
Recall that $\bar{G}_{(\lambda^*)}$ denotes the set of unitary representations of $G$ for which \eqref{eq:Glambda} holds.

\begin{corollary}
Suppose that $0< t < 1$.
Then $\bar{G}_{(t\rho)}$ is the collection of unitary representations of $G$ all of whose matrix coefficients lie in $\fnspace{L}^{q+}(G)$, where $q  = 2/(1-t)$.
\end{corollary}

\begin{proof}
Suppose that $\pi \in \bar{G}_{(t\rho)}$.
Then the $K$-finite matrix coefficients of $\pi$ are dominated by multiples of $\phi_{t\rho}$, and the integral formula for the Cartan decomposition and the estimates of Narayanan, Pasquale and Pusti \cite{{NPP14}} show that $\phi_{t\rho} \in \fnspace{L}^{q+}(G)$ when $q = 2/(1-t)$.
By \cite[Theorem 1.5]{SW18}, all matrix coefficients of $\pi$ lie in $\fnspace{L}^{q+}(G)$.

Conversely, we must show that if the matrix coefficients of $\pi$ lie in $\fnspace{L}^{q+}(G)$ where $q \geq 2$, then $\pi \in \bar{G}_{(t\rho)}$.
Now $\pi$ may be written as a direct sum of cyclic representations, each of which has a cyclic vector, and it will suffice to show that each cyclic component lies in $\bar{G}_{(t\rho)}$.
In other words, it suffices to suppose that $\pi$ has a cyclic vector.

Suppose then that $\pi$ has a cyclic vector $\theta$ and that $\lip\pi(\cdot)\theta,\theta\rip \in \fnspace{L}^{q+\epsilon}(G)$.
By \cite[Theorem 1.5]{SW18}, it follows that $\fnspace{A}_\pi(G) \subseteq \fnspace{L}^{q+2\epsilon}(G)$ and, by the closed graph theorem, that there is a constant $C$, possibly depending on $\pi$, $q$ and $\epsilon$, such that
\[
\norm{\lip\pi(\cdot)\xi,\eta\rip}_{q+2\epsilon}
\leq C \norm{\xi}_{\cH_\pi} \norm{\eta}_{\cH_\pi}
\qquad\forall \xi, \eta \in \cH_\pi.
\]
This condition passes to the irreducible unitary representations $\sigma$ that are weakly contained in $\pi$.
If we can show that each such $\sigma$ satisfies
\[
\mathcal{A} \lip \sigma(\cdot) \xi, \eta\rip
\leq \lnorm \xi \rnorm_{\Hil_\sigma} \lnorm \eta \rnorm_{\Hil_\sigma}
    \phi_{t\rho}
\qquad\forall \xi, \eta \in \cH_\sigma,
\]
then $\pi$ will have the same property.
Thus without loss of generality, we may suppose that $\pi$ is irreducible.

Suppose finally that $\pi$ is irreducible and $\theta \in \cH_\pi \setminus \{0\}$ such that $\lip\pi(\cdot)\theta,\theta\rip \in \fnspace{L}^{q+}(G)$.
It follows by a convolution argument that $\lip\pi(\cdot)\xi,\eta\rip \in \fnspace{L}^{q+}(G)$ for all $K$-finite vectors $\xi$ and $\eta$.
By a Sobolev embedding argument, as in \cite[Corollaire 2.2.4]{C79}, and the estimates for the spherical functions $\phi_{t\rho}$ when $0 < t < 1$, we deduce that
\[
\labs \lip\pi(k \exp(H) k')\xi,\eta\rip \rabs
\leq C \exp( - 2(q+\epsilon)^{-1} \rho(H))
\leq C' \phi_{(t+\epsilon)\rho}(\exp H)
\]
for all $H \in (\Lie{a}^+)^-$ and all $k, k' \in K$, and for all $\epsilon \in \R^+$.
We lose control of the constants as $\epsilon$ tends to $0$.

By Theorem B, it follows that
\[
\mathcal{A} \lip \sigma(\cdot) \xi, \eta\rip
\leq \lnorm \xi \rnorm_{\Hil_\sigma} \lnorm \eta \rnorm_{\Hil_\sigma}
    \phi_{(t+\epsilon)\rho}
\qquad\forall \xi, \eta \in \cH_\sigma,
\]
for all positive $\epsilon$, and letting $\epsilon$ tend to $0$ concludes the proof.
\end{proof}

\subsection*{Tensor products}
We get some information about tensor products.
\begin{corollary}
Take $\pi \in \bar{G}_{(\lambda)}$ and $\pi' \in \bar{G}_{(\lambda')}$, where $\lambda, \lambda' \in (\Lie{a}^+)\afterbar$.
Then $\pi \otimes\pi' \in \bar{G}_{(\lambda'')}$, where $\lambda'' \in (\Lie{a}^*)\afterbar$ and $\lambda'' \geq \lambda+\lambda'-\rho$.
\end{corollary}

\begin{proof}
The finite sums of products of matrix coefficients of $\pi$ and matrix coefficients of $\pi'$ forms a dense subset of the set of matrix coefficients of the tensor product $\pi \otimes\pi'$; Theorem B is used to pass from the dense subspace to the whole Hlbert space of $\pi \otimes\pi'$.
\end{proof}

\subsection*{Classification of unitary representations}
The following corollary of Corollary \ref{cor:minimal} is a minor observation on the Langlands classification of unitary representations.

\begin{corollary}
Let $\pi$ be an irreducible unitary representation $\pi$ of $G$, and $\{\lambda_j\}$ be the set of leading exponents in the asymptotic expansion of the matrix coefficients of $\pi$, as in Theorems \ref{thm:asymptotics} and \ref{thm:Langlands-Knapp-clsfctn}.
Let $Q$ be the parabolic subgroup determined by the set  $\Delta_Q$ of simple roots given by $\alpha_i \in \Delta_Q$ if and only if $\lip\Re(\lambda_j), \alpha_i\rip = 0$ for all $j$.
Then $\pi \in \bar{G}_{H,\temp}$ when $H = M_Q$ but not when $H = M_{Q'}$, where $Q'$ is a larger parabolic subgroup of $G$ than $Q$.
\end{corollary}

\subsection*{Representations on homogeneous spaces of $G$}
Benoist and Kobayashi \cite{BK15} considered the quasiregular representation $\pi$ of a semisimple group on functions on the homogeneous space $G/H$, and showed that the matrix coefficients of this representations satisfy pointwise estimates for a dense set of vectors; more precisely, they prove estimates of the form
\begin{equation}\label{eq:Benoist-Kobayashi}
\labs \lip \pi(\exp Y) \xi, \eta\rip \rabs
\leq C(\pi, \xi, \eta) \fn\expe^{-\rho_{\Lie{q}}^{\min} (Y)}
\qquad\forall Y \in \Lie{a}^+
\end{equation}
for vectors $\xi$ and $\eta$ in a dense subspace of $\fnspace{L}^2(G/H)$, and show that the decay term is best possible.
They then deduce $\fnspace{L}^{q+}(G)$ estimates for all matrix coefficients of $\pi$, where $q$ is an even integer, and hence conclude that certain representations cannot appear in the decomposition of the representations on $\fnspace{L}^2(G/H)$ into irreducible components.
Control of the constants that appear in \eqref{eq:Benoist-Kobayashi} is difficult, and so it is not obvious that the representations that appear in the decomposition satisfy similar decay estimates to the original representation.
In light of the estimate of Theorem  \ref{thm:NPP}, we can replace the estimate \eqref{eq:Benoist-Kobayashi} with
\[
\labs \lip \pi(\exp Y) \xi, \eta\rip \rabs
\leq C'(\pi, \xi, \eta) \fn\phi_{\rho_{\Lie{q}}-\rho} (\exp(Y)),
\]
and $\phi_{\rho_{\Lie{q}}-\phi}$  is hermitean because \eqref{eq:Benoist-Kobayashi} is sharp. 
We then use Theorem B to deduce sharper growth restrictions on the matrix coefficients of the representations that appear in the decomposition of $\pi$.

\subsection*{Restrictions of representations to subgroups}

Much as immediately above, if we restrict a unitary representation of a semisimple group $G$ to a closed semisimple (or reductive) subgroup $H$, then the pointwise estimates for $K$-finite matrix coefficients give rise to pointwise estimates for matrix coefficients of the restricted representation corresponding to a dense set of vectors.
These in turn imply estimates for all matrix coefficients of the restricted representation and the representations that appear in its decomposition into irreducibles.

\subsection*{Isolation of the trivial representation}
If $\pi$ in $\hat{G}$ does not weakly contain the trivial representation, then we may choose a spherical function $\phi_\lambda$ in Theorem B that vanishes at infinity in $G$, and $\lambda$ is in the interior of $\conv(W\rho)$.

If $G$ is a semisimple Lie group that does not contain normal subgroups locally isomorphic to $\group{SO}(1,n)$ or $\group{SU}(1,n)$, and $\pi$ is a representation of $G$ that does not strongly contain the trivial representation, then $\pi$ does not weakly contain the trivial representation. In this case, there is an element $\kappa$ of $(\Lie{a}^*)^+$ such that $\kappa \preceq \rho$ and every $K$-finite matrix coefficient of every unitary representation $\pi$ of $G$ without trivial subrepresentations may be dominated by a multiple of $\phi_{\kappa}$; moreover $\phi_\kappa$ decays exponentially at infinity.
More precisely,
\[
\sup_{k,k' \in K} \labs \lip \pi (k x k') \xi, \eta\rip\rabs
\leq \dim(\Span(\pi(K)\xi))^{1/2} \lnorm \xi \rnorm_{\Hil_\pi} \dim(\Span(\pi(K)\eta))^{1/2}  \lnorm \eta \rnorm_{\Hil_\pi} \phi_\kappa(x)
\]
for all $x \in G$ and all $k,k' \in K$, and for all $\xi, \eta \in \cH^K$ (the space of $K$-finite vectors in $\cH_\pi$).
In particular this holds even on the walls.
This answers a question of R.J. Zimmer (personal communication).

The identification of $\kappa$ was the work of several authors, including R.E. Howe \cite{Ho82} and H. Oh \cite{Oh02}.

\section{Afterthoughts}

The estimates here arise from deep results of Harish-Chandra, Langlands and others, combined with functional analytic techniques.
The results of \cite{CHH89} were used to prove that the algebras $\fnspace{E}_{(\lambda)}(G)$ have faithful representations, but what was actually used is the measure theoretic \emph{principe de majoration} of Herz \cite{Hz70}, not the harder estimates of \cite{C78}.

To explain heuristically why the functional analysis enables us to pass from results that hold in proper subcones of the cone $(\Lie{a}^*)^+$ to results that hold globally, we recall \eqref{eq:SW}:
\[
  \lnorm \pi(f) \rnorm
= \sup_{\xi \in \mathcal{H}_0}  \lim_{n \to \infty}
            \lip  \lpar \pi(f^* * f)^{(*n)}\rpar \xi, \xi \rip^{1/2n}  .
\]
For semisimple Lie groups, we have information on the behaviour of convolution powers of positive $\fnspace{L}^1(G)$-functions, going back to Oseledets \cite{Os68} and Guivarc'h \cite{Gu90}; this can be parlayed into information about convolution powers of arbitrary functions $f^**f$.
In the special case of powers of the heat kernel, this information is very precise: see \cite{AS92}.
The main point is that these powers $(f^**f)^{(*n)}$ concentrate in the radial directions, rather than spreading; indeed, in the Cartan decomposition, the bulk of the mass is near to $K a_n K$, where $a_n$ is a point in $A^+$ of the form $\exp(nH)$; $H$ depends on the ``initial function'' $f^* * f$.

The asymptotic behaviour of spherical functions shows that $\phi(\exp(nH))$ behaves like a finite sum of terms of the form
$p(nH) \expe^{(\alpha-\rho)(nH)}$, where $p$ is a polynomial.
As $(f^**f)^{(*n)}$ ``concentrates'' somewhat, we see heuristically that
\[
\lpar \int_G (f^* * f)^{(*n)} \phi_\lambda(x) \wrt x \rpar^{1/2n}
\]
should converge, and the limit should be expressible in terms of $\alpha$ and $H$.

There are global estimates for the spherical functions $\phi_\lambda$ (see Theorem \ref{thm:NPP}), so Theorem B provides global estimates for matrix coefficients that have positive features of both the asymptotic expansion and the $\fnspace{L}^{q+}(G)$ estimates; in particular, they hold everywhere in $G$.
For many irreducible representations, the spherical function parameter is sharp.
However, it is probably possible but quite difficult to prove sharper results.
In particular, for the group $\group{SL}(2,\R)$, estimates have been proved \cite{BCNT22} for the matrix coefficients of almost all tempered representations in which the polynomial term of $\phi_0$  does not appear.

\end{document}